\documentclass[12 pt]{amsart}
\usepackage{amssymb}
  
\usepackage{amsmath}
\usepackage{amssymb}
\usepackage{amsthm}
\usepackage{amscd}

\theoremstyle{plain}
\newtheorem{thm}{Theorem}[section]

\newtheorem{ques}[thm]{Question}

\theoremstyle{definition}

\newtheorem{dfn}[thm]{Definition}
\newtheorem{expl}[thm]{Example}
\theoremstyle{remark}
\newtheorem{rem}[thm]{Remark}

\newtheorem*{acknowledgement}{Acknowledgment}

\def\mm{\frak m}
\def\nn{\frak n}

\def\LR{\Longrightarrow}%

\def\Hom{\operatorname{Hom}}

\def\gr{\operatorname{gr}}

\begin{document}

\title{Upper bound of multiplicity of F-rational rings and F-pure rings}

\author{Craig Huneke}
\address{Department of Mathematics\\
University of Virginia\\
Charlottesville, VA 22904 USA}
\email{huneke@virginia.edu}

\author{Kei-ichi Watanabe}
\address{Department of Mathematics, College of Humanities and Sciences, 
Nihon University, Setagaya-ku, Tokyo 156--0045, Japan}
\email{watanabe@math.chs.nihon-u.ac.jp}

\thanks{The first author is partially supported by NSF
grant 1259142. The second author is partially supported by Grant-in-Aid for Scientific
Research 20540050
and Individual Research Expense of College of Humanity and Sciences, Nihon University.
We thank AIM for the opportunity to work together.}



\maketitle

\begin{abstract} This paper answers in the affirmative a question raised by Karl Schwede
concerning an upper bound on the multiplicity of F-pure rings.

\end{abstract}

\maketitle

\section{Introduction}

In the problem session of the workshop at AIM, August 2011, titled \lq\lq Relating 
Test Ideals and Multiplier Ideals", Karl Schwede, inspired by the work of Stefan 
Helmke \cite{He},  posed the following question.
\vskip 0.5cm 
Let $(R, \mm)$ be a Noetherian local ring of characteristic $p>0$ of dimension $d$ and 
embedding dimension $v$. 
Assume that  $R$ is F-pure. Then does the multiplicity $e(R)$ of $R$ always satisfy 
$$e(R) \le \left(\begin{array}{c} v\\ d \end{array}\right)\quad ?$$ 

We will prove this inequality is  true, and follows from a Brian\c con-Skoda type theorem. Our results can
be used to give an alternate proof of one of the main results of \cite{ST},  bounding the number of F-pure
centers in an F-pure F-finite local ring, which was one of the reasons for asking the motivating question. 
See Remark \ref{Fpurecenter}.
\par

\section{Preliminaries}

Let $(R,\mm)$ be either a Noetherian local ring or $R=\oplus_{n\ge 0} R_n$ be a 
graded ring finitely generated over a field $R_0=k$.    
We always assume that either $R$ contains a field of characteristic $p>0$ or 
$R$ is essentially of finite type over a field of characteristic $0$. 
{\bf We always assume that our ring $R$ is reduced.}

\begin{dfn} We denote by $R^{\circ}$ the set
of elements of $R$ that are not contained in any minimal prime
ideal. The \textit{tight closure} $I^*$ of $I$ is defined to be the ideal
of $R$ consisting of all elements $x \in R$ for which there exists
$c \in R^{\circ}$ such that $cx^q \in I^{[q]}$ for all large
$q=p^e$.
\end{dfn}

\begin{dfn} We say that a local ring $(R,\mm)$ is  \textit{F-rational} if 
 it is a homomorphic image of a Cohen-Macaulay ring and 
for every parameter ideal $J$ of $R$ we have $J^{*}=J$. 
It is known that F-rational rings are normal and Cohen-Macaulay.  
\end{dfn}

\begin{dfn} Assume that $R$ contains a field of characteristic $p>0$ and
$q=p^e$ be a power of $p$. 
\begin{enumerate}
\item For a power $q= p^e$ and ideal $I$ in $R$, we denote by $I^{[q]}$, 
the ideal generated by $\{a^q \;|\; a\in I\}$.
\item We write $R^{1/q}$  
then we say that 
$R$ is {\it F-pure} if for every $R$ module $M$, 
the natural map $M= M\otimes_R R \to M\otimes _R R^{1/p}$, sending $x\in M$ 
to $x\otimes 1$ is injective. 
\item Let $I$ be an ideal of $R$ and $x\in R$.
If $R$ is F-pure and if $x^q \in I^{[q]}$, then $x\in I$.  This follows from (2) 
if we put $M= R/I$.  
\end{enumerate}
\end{dfn}

\section{The main results}

The following theorem is our main result in this article.

\begin{thm}\label{e(R)} Let $(R,\mm)$ be a Noetherian local ring 
with $\dim R =d$ and embedding dimension $v$.  
Then, 
\begin{enumerate}
\item If $R$ is a rational singularity or F-rational, then $e(R) \le \left(\begin{array}{c} v-1\\ d-1 
\end{array}\right).$
\item If $R$ is F-pure, then $e(R) \le \left(\begin{array}{c} v\\ d \end{array}\right)$.
\end{enumerate}
\end{thm}
 
This theorem easily follows from the following theorem. We recall that a {\it reduction} $J$ of
an ideal $I$ is an ideal $J\subset I$ such that for large $n$,  $I^n = JI^{n-1}$.  A {\it
minimal reduction} is a reduction minimal with respect to inclusion. A fundamental fact is
that in a Noetherian local ring of dimension $d$ having infinite residue class field, minimal
reductions are always generated by at most $d$ elements. See \cite[Section 8.3]{HS}.

\begin{thm}\label{BS}
 Let $(R, \mm)$ be a Noetherian local ring 
with $\dim R =d$ and let $J\subset \mm$ be a minimal reduction of $\mm$.
\begin{enumerate}
\item If $R$ is a rational singularity or F-rational, then $\mm^d \subset J$.  
\item If $R$ is F-pure, then $\mm^{d+1} \subset J$.  
\end{enumerate}
\end{thm}

\begin{proof}  The statement (1) is well known and follows from Brian\c con-Skoda 
type theorem (cf. \cite{HH}, \cite{LT}).  \par
For the statement (2) we will prove the following statement. \par

Assume $R$ is F-pure and $I$ is an ideal generated by $r$ elements, 
which contains a nonzerodivisor, 
then $\overline{I^{r+1}}\subset I$.  This is sufficient to prove \ref{BS} since 
$\mm^{d+1} \subset \overline{\mm^{d+1}}=\overline{J^{d+1}}$.
\par
Now, take $x\in  \overline{I^{r+1}}$.  Then we can take $c\in R^{\circ}$ such that 
for sufficiently large $N, c x^N \in I^{(r+1)N}$.  Then $c x^N \in c (I^{(r+1)N} : c)$. 
The latter is contained in $cR \cap I^{(r+1)N} $ and by  the Artin-Rees Lemma, 
there exists $k$ such that 
$cR \cap I^{(r+1)N} \subset c I^{(r+1)N - k }$ for sufficiently large $N$. 
Now, we have shown that $c x^N \in  c I^{(r+1)N - k }$.  
Note that $I^{rq} \subset I^{[q]}$.   
Taking sufficiently large $N =q =p^e$ and noting that $c$ is a non zero divisor, 
we get $x^q \in I^{[q]}$. Since $R$ is F-pure, we get $x \in I$. 
This was proved in \cite{Hu},  Proposition 4.9 when $R$ is Cohen-Macaulay.
\end{proof}

It is easy to prove \ref{e(R)}  using \ref{BS}. 
 
\begin{proof}[Proof of \ref{BS} $\LR$ \ref{e(R)}]  We have the following inequality  and 
the equality holds if and only if $R$ is Cohen-Macauly (cf. \cite{BH}, Corollary 4.7,11). 

$$\leqno{(3.1.1)} \qquad e(R) \le l_R(R/J)$$

So, it suffices to show that $l_R(R/J)$ is bounded by  the right-hand side of the 
inequalities in \ref{e(R)}.  Now, let $x_1,\ldots , x_d, y_1,\ldots , y_{v-d}$ be minimal 
generators of $\mm$ with $J = (x_1,\ldots , x_d)$.  Then $R/J$ is generated by 
the monomials of $y_1,\ldots ,y_{v-d}$ of degree $\le d-1$ (resp. degree $\le d$) 
in case (1) (resp. case (2)) by \ref{BS}.  It is easy to see that the number of 
monomials of $y_1,\ldots ,y_{v-d}$ of degree $\le d-1$ (resp. degree $\le d$) 
is $\left(\begin{array}{c} v-1\\ d-1 \end{array}\right)$ (resp.  
$\left(\begin{array}{c} v\\ d \end{array}\right)$).  
\end{proof}

\begin{rem}  Assume we have equality in \ref{e(R)} (1) or (2).  Then 
$R$ is Cohen-Macaulay since we must have equality in (3.1.1), too.
Moreover, since the associated graded ring of $R$ has the same 
embedding dimension and multiplicity with $R$, $\gr_{\mm}(R)$ 
is also Cohen-Macaulay in this case.
\end{rem}

\begin{rem}\label{Fpurecenter} Another class of singularities which would
be natural to consider are  F-injective (respectively, Du Bois) singularities.
 So, it is natural to ask if \ref{BS} is true if we assume $R$ is F-injective or a  
 Du Bois singularity. We do not know the answer. Also, we point out that our \ref{BS} (2) gives another 
 proof of one of the main theorems of  \cite{ST} concerning the number of F-pure
 centers.  Actually, in Theorem 5.10 of \cite{ST}, if $Q_i$ ($1\le i\le N$) are prime 
 ideals of dimension $d$ of $R$ such that every $R/Q_i$ is F-pure, then   
 certainly the number $N$ is bounded by the multiplicity of 
 $R/[\cap_{i=1}^N Q_i]$ and the latter is F-pure. Thus we can apply our 
 \ref{BS}(2).
 We thank the anonymous referee for suggesting these points.
 \end{rem}

\section{Actual Upper Bound}

The upper bound in \ref{e(R)} (2) is taken by the following example. 

\begin{expl} Let $\Delta$ be a simplicial complex on the vertex set $\{ 1,2,\ldots ,v\}$, 
whose maximal faces are all possible $d-1$ simplices.  Then the Stanley-Reisner ring 
$R=k[\Delta]$ has dimension $d$ and  $e(R)=  \left(\begin{array}{c} v\\ d \end{array}\right)$.
Note that Stanley-Reisner rings are always F-pure. 
\end{expl}

\begin{rem}  (1)  Are there other examples where we have equality in \ref{e(R)} (2) 
if $v \ge d+2$ ?  
It is shown in \cite{GW} that in the case of $d=1$, this is the only example if we assume 
$(R, \mm)$ is complete local ring with algebraically closed residue field. \par

\medskip

(2) It is natural to ask if there  
are examples where we have equality in \ref{e(R)} (1) 
if $v \ge d+2$ and $d\ge 3$. If $d=2$, we have always $e(R) = v-1$
(cf. \cite{Li}).  See remark (4) below.  \par

\medskip

(3) It is not difficult to see that the examples which attain the maximal value in Theorem \ref{e(R)} (1) must
be generated by $\left(\begin{array}{c} v-1\\ d 
\end{array}\right)$ elements of degree $d$,
and have defining ideal with a linear resolution.  In fact, let $R$ be  a rational singularity or F-rational
having maximal multiplicity $e(R) = \left(\begin{array}{c} v-1\\ d-1 
\end{array}\right),$ where the dimension of $R$ is $d$ and $v$ is the embedding dimension.
Let $A$ be a general Artinian reduction of $R$, that is to say, $A$ is $R$ modulo a general linear
system of parameters (we assume infinite field here).  Then as the proof of the main theorem
shows, $A\cong k[x_1,...,x_c]/(x_1,...,x_c)^d$, where $c = v-d$ is the embedding codimension of $R$.
It is well-known that $(x_1,...,x_c)^d$ has a linear resolution. Since $R$ must be Cohen-Macaulay,
its defining ideal will also have a linear resolution. 
Observe that the $a$-invariant of $R$ must then be $-1$, since the dimension of $R$ is $d$, and
the socle degrees of the Artinian reduction $A$ are all $d-1$. 

The converse will often be true; if $S$ is a polynomial ring,
and $R'$ a graded F-rational quotient ring (or rational singularity) whose defining ideal $J$ has a linear resolution, then
provided the common degree of the generators of $J$, say $d$, is at most the dimension of $R'$,
then one should be able to cut $R'$ down by general linear forms (at least over an algebraically closed field) to a ring
$R$ such that the dimension of $R$ is exactly $d$ and $R$ is F-rational or has rational singularities on the punctured spectrum
(see, for example, \cite{SZ}).
But then it is enough to check that the a-invariant is negative to prove $R$ is F-rational \cite{Wa, F}. This
follows from the fact $R$ has a linear resolution over the polynomial ring obtained from $S$ by
cutting with the same general linear forms.  The a-invariant of $R$ will be $-1$.  The multiplicity of
$R$ will be exactly  $\left(\begin{array}{c} v-1\\ d-1 
\end{array}\right)$, where $v$ is the embedding dimension of $R$. 

\medskip

(4) To see an explicit  example as in (3), with even an isolated singularity, consider the ideal of maximal minors
of a generic $r$ by $s$ matrix  $X$ over an algebraically closed field of characteristic $0$.  Assume that $2r\leq s+3$.
Let $S$ be the ambient polynomial ring, and let $R' = S/I$, where $I$ is generated by the maximal minors of the
generic matrix $X$. The singular locus of $R'$ is defined by the image of the $r-1$ size minors of $X$, which has
height in $R'$ exactly $(r-(r-1)+1)(s-(r-1)+1)-(s-r+1) = 2(s-r+2)-(s-r+1)= s-r+3$.  Since $r\leq s-r+3$ by assumption,
we can reduce $R'$  modulo $rs-s+2r-1$ general linear forms to reach an $r$-dimensional  ring $R$ with an isolated singularity (see \cite[Satz 5.2]{F2}), defined
by the maximal minors of an $r$ by $s$ matrix $Y$ of linear  forms such that the $r$ by $r$ minors have generic height. 
The $a$-invariant of $R$ is $-1$, so by the results of \cite{F} and \cite{Wa}, $R$ has a rational singularity. The multiplicity of 
$R$ is exactly  $\left(\begin{array}{c} v-1\\ d-1 
\end{array}\right)$, where $v$ is the embedding dimension and $d$ is the dimension of $R$.
\end{rem}

It seems likely that the following question will have a positive answer:

\begin{ques}  Assume that $(R,\mm)$ is a rational singularity or F-rational with dimension $d$ and
embedding dimension $v$ with maximal possible multiplicity (as in Theorem \ref{e(R)}).  
  Then is the associated graded ring $gr_{\mm}(R) $ Cohen-Macaulay having a defining ideal with linear
resolution?\end{ques}

\section{Case of Gorenstein Rings}

If $R$ is Gorenstein, the upper bound is largely reduced by the duality. We prove:

\begin{thm}\label{Gor} Let $(R,\mm)$ be a Gorenstein Noetherian local ring 
with $\dim R =d$ and embedding dimension $v$.

\begin{enumerate}
\item If $R$ is a rational singularity or F-rational with $\dim R = 2r +1$, then 
  $e(R) \le \left(\begin{array}{c} v-r - 1\\  r \end{array}\right)+ 
  \left(\begin{array}{c} v-r-2\\ r-1 \end{array}\right).$
  
\item If $R$ is a rational singularity or F-rational with $\dim R = 2r$, then 
  $e(R) \le 2 \left(\begin{array}{c} v-r - 1\\  r -1.\end{array}\right)$  

\item If $R$ is F-pure with $\dim R= 2r+1$, then $e(R) \le 
2 \left(\begin{array}{c} v-r-1 \\  r 
\end{array}\right)$.
\item If $R$ is F-pure with $\dim R = 2r$, then 
  $e(R) \le \left(\begin{array}{c} v-r \\  r 
\end{array}\right)+ \left(\begin{array}{c} v-r-1\\ r-1 
\end{array}\right).$
\end{enumerate}
\end{thm}

\begin{proof} We will prove the first statement. All the others follow in exactly the same manner.  We may assume that the residue
field is infinite. Let $J$ be a minimal reduction of the maximal ideal of $R$, and let $B = R/J$, an Artinian Gorenstein ring. By
Theorem \ref{BS}, we know that $\mm^d\subset J$, so that if $\nn$ denotes the maximal ideal of $B$, $\nn^{2r+1} = 0$. 
We estimate the length of the $B$ (which is the multiplicity of $R$). First, observe that since $B$ is Gorenstein, 
$l(B/\nn^t)$ is the same as the length of the Matlis dual module, $\Hom(B/\nn^t,B)$, which is equal to the length of $(0:\nn^t)$.
Therefore,
$$e(R) = l(B) = l(\nn^r) + l(B/\nn^r) \leq l((0:\nn^{r+1})) + l(B/\nn^r) = l(B/\nn^{r+1}) + l(B/\nn^r),$$
where the inequality follows because $\nn^r\subset (0:\nn^{r+1})$. 
Now $B$ has embedding dimension $v-(2r+1)$. As before, $B$ a homomorphic image of
a polynomial ring in  $v-(2r+1)$ variables. As above, the length of a polynomial ring in this many
variables modulo the $(r+1)$st power of its maximal ideal is   $\left(\begin{array}{c} v-r - 1\\  r \end{array}\right)$  , while modulo the $r$th power is   $\left(\begin{array}{c} v-r-2\\ r-1 \end{array}\right)$, giving the statement of part (1). \end{proof}

\begin{rem} Again, the upper bound in (3), (4) is taken by the Stanley-Reisner 
ring of \lq\lq Cyclic Polytopes" (cf. \cite{St}).  
\end{rem}

\begin{ques} As was the case in the last section, it is reasonable to ask the following:
suppose that $(R,\mm)$ is Gorenstein with rational singularity (or F-rational) having the maximal
possible multiplicity given the dimension and embedding dimension. Then
is $gr_{\mm}(R)$ Gorenstein with "symmetric linear"
resolution (i.e., pure resolution with degree sequence $(n, n+1, \ldots , n+s, 2n+s)$?\end{ques}

\begin{acknowledgement}
We are grateful to Naoki Terai for suggesting  cyclic polytopes as examples. We also
thank Karl Schwede and Wenliang Zhang for valuable communications. We also thank AIM
for hosting the workshop where this question was first discussed, and the main result was proved.
\end{acknowledgement}

\end{document}